\documentclass[11pt]{amsart}
\usepackage{amsmath,amsfonts, mathtools,amssymb,latexsym,amssymb,amsthm,adjustbox,mathrsfs,amsbsy, bm,tikz-cd,indentfirst, makecell,graphicx, verbatim, calc, enumerate,xy, hyperref}
\usepackage[top=1in,bottom=1.3in,left=1.1in,right=1.1in]{geometry}
\usepackage{fancyhdr}
\hypersetup{colorlinks=true,
    linkcolor=blue,
    filecolor=magenta,      
    urlcolor=cyan,}
\theoremstyle{plain}
\newtheorem{theorem}{Theorem}[section]
\newtheorem{lemma}[theorem]{Lemma}

\newtheorem{corollary}[theorem]{Corollary}
\newtheorem{proposition}[theorem]{Proposition}

\newtheorem{definition}[theorem]{Definition}

\newtheorem{point}[theorem]{}

\newcommand{\D}{\mathcal{D}}
\newcommand{\m}{\mathfrak{m}}

\newcommand{\n}{\mathfrak{n}}

\newcommand{\grade}{\operatorname{grade}}
\newcommand{\injdim}{\operatorname{injdim}}

\newcommand{\Ass}{\operatorname{Ass}}
\newcommand{\height}{\operatorname{height}}

\newcommand{\Supp}{\operatorname{Supp}}

\newcommand{\Der}{\operatorname{Der}}
\newcommand{\Spec}{\operatorname{Spec}}

\newcommand{\Hom}{\operatorname{Hom}}
\newcommand{\Mod}{\operatorname{Mod}}

\newcommand{\Ext}{\operatorname{Ext}}

\usepackage{graphicx}
\usepackage{subfigure}
\usepackage{epsfig}
\usepackage{epstopdf}
\usepackage{float} 
\theoremstyle{plain}

\usepackage{xcolor}
\usepackage{titlesec}

\titleformat{name=\section}{}{\thetitle.}{0.5em}{\centering\scshape\large}

\begin{document}
\title{\large \textbf{Local cohomology modules of a regular affine domain}}
\author{\textsc{Sayed Sadiqul Islam}}
\address{Department of Mathematics, IIT Bombay, Powai, Mumbai 400076, India}
\email{ssislam1997@gmail.com, 22d0786@iitb.ac.in}
\date{\today}

	\subjclass{Primary 13D45; Secondary 13N10, 13H10}
	\keywords{D-modules, local cohomology, injective resolutions}

\begin{abstract}
	For a Noetherian commutative ring $R$, let $H^i_I(R)$ be the $ i$-th local cohomology module of $R$ with respect to $I$. In \cite{Hel-08}, Hellus posed the question of identifying rings $R$ such that $\injdim_R H^i_I(R)=\dim_R(\Supp_R H^i_I(R))$. In this paper, we show that  a regular affine domain over a field of characteristic $0$ satisfies this condition. In fact, we prove that $\injdim_RH^i_I(R)\geq \dim_R(\Supp_R H^i_I(R))-1$ when $R$ is a differentiably admissible $K$-algebra. Indeed, we establish both of these conclusions for a substantially broad class of functors known as Lyubeznik functors. We also prove that if $R$ is a polynomial ring over a differentiably admissible $K$-algebra, then $\Ass_R H^i_I(R)$ is finite for all $i\geq 0$ and for every ideal $I$ of $R$.
	\end{abstract}
	 \maketitle
\section{Introduction}
Throughout this paper, $R$ is a commutative Noetherian ring with unity. For a locally closed subscheme of $\Spec(R)$, let $H^i_Y(M)$ denotes the $i$-th local cohomology module of $M$ supported in $Y$. When $Y$ is closed in $\Spec(R)$, defined by an ideal $I$ of $R$, $H^i_Y(M)$ is  denoted by $H^i_I(M)$. 

Local cohomology modules are not finitely generated in general. For example, if $(R,\mathfrak{m})$ is a local ring of dimension $d\geq 1$, then $H^d_\mathfrak{m}(R)$ is never finitely generated. The lack of finite generation has turned attention of researchers to study other type of finiteness properties. In this direction, one deals with the following finiteness properties of local cohomology modules over certain regular rings.
\begin{enumerate}
        \item $H^j_\m(H^i_I(R))$ is injective, where $\m$ is maximal ideal of $R$.
        \item $\injdim_R H^i_I(R)\leq \dim_R( \Supp_R H^i_I(R))$.
        \item  The set of associated primes of $H^i_I(R)$ is finite.
        \item  All the Bass numbers of $H^i_I(R)$ are finite.
    \end{enumerate}
The above properties were first proved by Huneke and Sharp in \cite{HS-93} for regular rings of characteristic $p>0$ using Frobenius map. Later Lyubeznik in his remarkable paper \cite{Lyu-93} used the theory of $\D$-modules and proved properties (1), (2) and (4) for large class of functors (known as Lyubeznik functors) over regular rings containing a field of characteristic 0. He also proved (3) under the assumption that $R$ is local. Later in \cite{Lyu-97} Lyubeznik developed the theory of $F$-modules over regular rings containing a field of characteristic $P>0$ and proved the results of Huneke and Sharp for Lyubeznik functor. Recently,  Bhatt et al. in \cite{BBLSZ-14} proved property (4) for smooth $\mathbb{Z}$-algebras.

In \cite{Hel-08}, Hellus deals with following question.
\medskip

\noindent
\textbf{Question:}
When does $\injdim_R H^i_I(R)=\dim_R(\Supp_R H^i_I(R))$ hold?
\medskip

He also provided a positive answer to this question when $R$ is  either Noetherian regular local ring containing a field with $H^i_I(R)$ a $I$-cofinite or $R$ is Noetherian local Gorenstein ring with $H^n_I(R)=0$ for $n\neq i$; see \cite[Corollary 2.6, Remark 2.7]{Hel-08}. Hellus further provided examples of $R$ where the equality does not hold; see \cite[Example 2.9, 2.11]{Hel-08}.

Let $R$ be the polynomial ring $K[x_1,\ldots,x_d]$ where $K$ is a field of characteristic $0$ and let $\mathcal{T}$ be a Lyubeznik functor in the category of $R$-modules. In \cite{Put-14}, Puthenpurakal proved that if $c=\injdim_R \mathcal{T}(R)$, then $\mu_c(P,\mathcal{T}(R))=0$ for each non-maximal prime ideal $P$ of $R$. As a consequence, he proves that  $\injdim_R(\mathcal{T}(R))=\dim_R(\Supp_R \mathcal{T}(R))$. Later in \cite{WZ-17}, Zhang proved a stronger version of these results by establishing it for holonomic $\mathcal{D}_K(R)$-modules. Puthenpurakal also asked if these results can be extended to positive characteristic. In \cite{WZ-17}, Zhang gives positive answer to these question as well by proving $\injdim_R M=\dim_R(\Supp_R M)$ for $F_R$-finite, $F_R$-module $M$ where $R$ is a regular ring of finite type over an infinite field of characteristic $p>0$.

Taking these results as motivation, we now state the results proved in this paper.\medskip

\noindent
\textbf{Theorem A.} (Theorem \ref{P is maixmal ideal main th})\phantomsection\label{P is maximal ideal at inj dim}
 \textit{Let $K$ be a field of characteristic $0$ and $R$ be a regular affine domain over $K$. Let $\injdim_R\mathcal{T}(R)=c$ for some Lyubeznik functor $\mathcal{T}$ on $\Mod(R)$. If the Bass number $\mu_c(P,\mathcal{T}(R))>0$ for some prime ideal $P$ of $R$, then $P$ is a maximal ideal of $R$.}
 \medskip

 Note that if $R$ is a regular $K$-algebra and $\mathcal{T}$ is a Lyubeznik functor then $\injdim_R \mathcal{T}(R)\leq \dim_R (\Supp_R\mathcal{T}(R))$, see \cite[Theorem 3.4]{Lyu-93} and \cite[Theorem 1.4]{Lyu-97}. Surprisingly, as a direct implication of \hyperref[P is maximal ideal at inj dim]{Theorem A}, we establish that the equality holds for regular affine domains over a field of characteristic $0$.\medskip
 
 \noindent
\textbf{Corollary B.} (Corollary \ref{proof of the corollary})\label{Corollary about inj dim}
  \textit{Let $K$ be a field of characteristic $0$ and $R$ be a regular affine domain over $K$. Let $\mathcal{T}$ be a Lyubeznik functor on $\Mod(R)$. Then, $\injdim_R \mathcal{T}(R)=\dim_R(\Supp_R \mathcal{T}(R))$.}
 \medskip
 
 Effort has also been made to obtain lower bound for the injective dimension of $F$-module and $\D$-module. M. Dorreh in \cite{MD-16} proved that if $M$ is either a $F$-finite, $F$-module over a regular local ring of characteristic $p>0$ or if a module of the form $H^i_I(R)_g$ over a regular ring of characteristic $0$, then $\injdim_R M\geq \dim_R(\Supp_R M)-1$.
Later, Zhang and Switala in \cite{SZ-19} proved same result for holonomic $\D$-modules over formal power series ring over a field of characteristic $0$ and for $F$-finite, $F$-modules over regular rings of characteristic $p>0$. When the regular ring  does not contain a field, the bounds on injective dimension of local cohomology modules are different. In \cite{CZ-98}, Zhou proved that  $\injdim_RH^i_I(R)\leq \dim_R(\Supp_R H^i_I(R))+1$ when $R$ is an unramified regular local of mixed characteristic and $I$ is an ideal of $R$.

The next result we prove is the following.\medskip
 
 \noindent
\textbf{Theorem C.} (Theorem \ref{proof of bound on inj dim}) \phantomsection \label{lower bound on injective dim}
\textit{Let $R$ be a differentiably admissible $K$-algebra where $K$ is a field of characteristic $0$. Let $\mathcal{T}$ be a Lyubeznik functor on $\Mod(R)$. Then $$\injdim_R\mathcal{T}(R)\geq \dim_R(\Supp_R \mathcal{T}(R))-1.$$}

 As mentioned at the very beginning of the Introduction, substantial progress has been made toward identifying rings for which local cohomology modules have only finitely many associated primes.
  In fact, in \cite{Lyu-99} Lyubeznik conjectured the following.\medskip

\noindent
\textbf{Conjecture:}
Let $R$ be a regular ring and $I$ be an ideal of $R$. Then for each $i\geq 0$, the set $\Ass_R H^i_I(R)$ is finite.
\medskip

In general, the above problem has a negative answer when $R$ is not regular. Singh gave the first example of a singular ring $R$ having an ideal $I$ such that $\Ass_R H^i_I(R)$ is infinite; see \cite{AS-00}. Further negative examples appear in Katzman's work \cite{Kat-02} and Singh-Swanson's \cite{Ass-04}.

Now we state the final result of this paper.
\medskip 

\noindent
\textbf{Theorem D.} (Theorem \ref{proof of finite ass prime for diff ad alg}) \phantomsection \label{finiteness of associated primes}
 \textit{Let $R$ be a differentiably admissible $K$-algebra and $S$ be the polynomial ring $R[x_1,\ldots,x_m]$. If $I$ is an ideal of $S$, then the set $\Ass_S H^i_I(S)$ is finite.}
 \medskip

We now describe in brief the contents of this paper. The paper is organized in five sections. Section 2 introduces the necessary preliminaries and key supporting results. Section 3 is devoted to the proof of \hyperref[P is maximal ideal at inj dim]{Theorem A} and as an interesting consequence we derive \hyperref[Corollary about inj dim]{Corollary B}. In section 4, we prove \hyperref[lower bound on injective dim]{Theorem C}. Finally, the last section proves  \hyperref[finiteness of associated primes]{Theorem D}.
\section{Preliminaries}
This section reviews the theory of $\D$-modules, introducing necessary definitions and results that will be referenced throughout the paper.
\begin{point}\normalfont
\textbf{Rings of differential operators:}
Let $R$ be a commutative $A$-algebra. The ring of $A$-linear differential operators of $R$ is a subring $\D_A(R)\subseteq \Hom_{A}(R,R)$ whose elements are defined inductively as follows: the differential operators of order zero are defined by the multiplication by elements of $R$, i.e. $\D^0(R)\cong R$. Suppose we have defined operators of order $< n$. An operator $\delta \in \Hom_{A}(R,R)$ is of order less than or equal to $n$ if $[\delta,r]=\delta r-r\delta$ is an operator of order less than or equal to $n-1$ for all $r\in R$. We have a filtration $\D_A^0(R)\subset \D_A^1(R)\subset\ldots$ and the ring of differential operators is defined as
$$\D_A(R)=\bigcup_{m\geq 0} \D_A^m(R).$$

 Let $M$ be a $\D_A(R)$-module and $f\in R$. Then $M_f$ is also a $\D_A(R)$-module such that the natural map $M\rightarrow M_f$ is a morphism of $\D_A(R)$-modules. Let $I\subseteq R$ be an ideal of $R$ generated by $\underline f= f_1,f_2,\ldots,f_r\in R$, and let $M$ be an $R$-module. Then the \v Cech complex of $M $ with respect to  $\underline f$ is defined by 
 $$
  \Check{C}^{\bullet}(\underline f,M):\  0\rightarrow M\rightarrow \bigoplus_iM_{f_i}\rightarrow \bigoplus_{i,j}M_{f_if_j}\rightarrow\ldots \rightarrow M_{f_1\ldots f_r}\rightarrow 0
 $$
 where the maps on every summand are localization map upto a sign. The local cohomology module of $M$ with support on $I$  is defined by 
 $$H_I^i(M)=H^i( \Check{C}^{\bullet}(\underline f,M))$$
 Therefore, it follows that every local cohomology module over a $\D_A(R)$-module is again a $\D_A(R)$-module.
 \end{point}
\begin{point}
    \normalfont
    \textbf{Rings of differentiable type:}   We say that an associative ring $R$ is filtered if there exists an ascending filtration $\Sigma_0\subset\Sigma_1\subset\Sigma_2\ldots$ of additive subgroups such that $1\in \Sigma_0$, $\bigcup \Sigma_i=R$ and $\Sigma_i\Sigma_j\subset\Sigma_{i+j}$ for every $i$ and $j$. We denote by $gr^{\Sigma}(R)$ the associated graded ring which is defined as $$gr^{\Sigma}(R):=\Sigma_0\oplus \Sigma_1/\Sigma_0\oplus \Sigma_2/\Sigma_1\oplus\ldots.$$
    $R$ is said to be a ring of differentiable type if $gr^{\Sigma}(R)$ is commutative Noetherian regular with unity and each graded maximal ideals have same height \cite[Definition 1.1.1]{MM-91}.
  
    Let $R$ be a filtered ring of differentiable type with filtration $\{\Sigma_i\}$. A filtration $\{\Gamma_i\}$ compatible with $\{\Sigma_i\}$ on a left $R$-module $M$ is called good if the associated graded $gr^{\Sigma}(R)$-module $gr^{\Gamma}(M)$ is finitely generated. It is well known that a left $R$-module $M$ can be equipped with good filtration if and only if it is finitely generated and the Krull dimension of $gr^{\Sigma}(R)$-module $gr^{\Gamma}(M)$ does not depend on the choice of the good filtration $\Gamma$ of $M$. This number is called Bernstein dimension of $M$ and denoted by $d(M)$.
\end{point}
Let us recall the following result from \cite{MM-91}.
\begin{theorem}\cite[Theorem 1.2.2]{MM-91}
    Let $R$ be a ring of differentiable type. For a nonzero finitely generated left or right $R$-module $M$, the following holds
    $$\grade_R(M)+d(M)=\dim (gr^{\Sigma}(R))$$ where  $$\grade_R(M)=\operatorname{inf}\{j\geq 0: \Ext_R^i(M,R)\neq 0\}.$$
\end{theorem}
In particular, $d(M)\geq \dim (gr^{\Sigma}(R))-\operatorname{w.gl.dim}(A)$ where $\operatorname{w.gl.dim}(A)$ denote the weak global dimension of $A$. If this is an equality (or if $M=0$), then $M$ is said to be holonomic or module in the Bernstein class.

Now we define what is called a differentiably admissible algebras over a field of characteristic $0$. These algebras were introduced by N\'{u}\~{n}ez-Betancourt in \cite{NB-13}.
\begin{definition}
   Let $K$ be  a field of characteristic $0$. We call a $ K$-algebra $R$ differentiably admissible if it is  Noetherian, regular with unity and satisfies the following conditions:
        \begin{enumerate}[\rm (1)]
            \item $R$ is equi-codimensional of dimension $n$, i.e, the height of any maximal ideal is equal to $n$.
            \item  Every residual field with respect to a maximal ideal is an algebraic extension of $K$.
            \item $\Der_K(R)$ is a finitely generated projective $R$-module of rank $n$ and $R_\mathfrak m\otimes_R\Der_K(R)=\Der_K(R_\mathfrak m)$ for each maximal ideal $\mathfrak m$ of $R$.
        \end{enumerate}
\end{definition}
This hypothesis of the above definition is inspired by a result of Mebkhout and Narv\'{a}ez-Macarro; see \cite[1.1.2]{MM-91}. There $R$ is a commutative Noetherian ring containing a field $K$ of characteristic 0 and satisfies hypothesis (1) and (2) of the definition but instead of condition (3), there exist $K$-linear derivations $\partial_1,\ldots,\partial_n\in \Der_K(R)$ and $a_1,\ldots,a_n\in R$ such that $\partial_i(a_j)=\delta_{ij}$. In this way, the algebras considered  by N\'{u}\~{n}ez-Betancourt  includes a larger class of rings; see
\cite[Proposition 2.6]{NB-13}. 

It is clear that the localization of a differentiably admissible $K$-algebra  at a maximal ideal and the completion of a differentiably admissible $K$-algebra at a maximal ideal are again a differentiably admissible $K$-algebra. If $R$ is a differentiably admissible $K$-algebra of dimension $d$, then the power series ring $R[[y]]$ is also a admissible $K$-algebra of dimension $d+1$ \cite[Proposition 5.5]{MJB-21}. It is worth mentioning that the property does not carry over polynomial rings (for example, see \cite[Page 28]{MJB-21}). We note that any equidimensional regular finitely generated $K$-algebra is also differentiably admissible \cite[Proposition 1.2.4.1]{MM-91}. Interested readers are referred to \cite{Put-18} and \cite{SP-24} for some recent examples. We also note that the authors in \cite{ACR-17} extend most of Lyubeznik's finiteness properties for $\D$-modules over formal power series rings to differentiably admissible $K$-algebras.

 Let $R$ be a differentiably admissible $K$-algebra and $\D$ be the corresponding ring of differential operator on $R$. We refer \cite{Lyu-93} for the definition of Lyubeznik functor, but we note that the composition of local cohomology functors of the form $ H^{i_1}_{I_1}(H^{i_2}_{I_2}(\dots H^{i_r}_{I_r}(-)\dots)$ are significant examples of Lyubeznik functors. The following result is important for us.
\begin{theorem}\cite[Theorem 4.4]{NB-13}\label{finite associated prime for diff adm alg}
    Let $R$ be a differentiably admissible $K$-algebra and let $\mathcal{T}$ be a Lyubeznik functor on $\Mod(R)$. Then, $\Ass_R\mathcal{T}(R)$ is finite. In particular, this holds for $H^i_I(R)$ for every $i\geq 0$ and ideal $I$ of $R$.
\end{theorem}
 \section{Holonomic D-modules over differentiably admissible algebras}
Throughout this section, we continue to assume $K$ to be a field of characteristic $0$. In this section, we prove \hyperref[P is maximal ideal at inj dim]{Theorem A} and as an interesting consequence we derive \hyperref[Corollary about inj dim]{Corollary B}. Before that we need some results about holonomic $\D$-modules over differentiably admissible $K$-algebra. 
 
We now recall the following result from \cite{Mac-14}. 
\begin{theorem}\cite[Proposition 1.2.3.9]{Mac-14}\label{localization at a maximal of diff admissible alg}
    Let $R$ be a  differentiably admissible $K$-algebra. For any maximal ideal $\m$ of $R$, the canonical maps $$R_{\m}\otimes_R \D(R,K)\rightarrow \D(R_{\m},K)\leftarrow \D(R,K) \otimes_R R_{\m}$$ are isomorphisms.
   \end{theorem}
    We recall a result from \cite{ACR-17} that we need to prove our result.
    \begin{proposition}\cite[Proposition 2.5]{ACR-17}\label{completion preserves holonomicity}
    Let $(R,\m,K)$ be local differentiably admissible $k$-algebra. If $M$ is a $\D$-module, then $\widehat{R}\otimes_R M$ is a $\mathcal{D}(\widehat{R}, K)$-module. Moreover, $M$ is holonomic $\D$-module if and only if $\widehat{R}\otimes_R M$ is holonomic $\mathcal{D}(\widehat{R}, K)$-module.
\end{proposition}
Similarly, we have the following.
\begin{proposition}\label{localization preserves holonomicity}
     Let $R$ be a  differentiably admissible $K$-algebra. If $M$ is a $\D$-module, then $M_\m$ has a $D(R_\m,K)$-module structure for every maximal ideal $\m$ of $R$. Moreover, $M$ is  holonomic $\D$-module if and only if $M_\m$ is holonomic as $\mathcal{D}(R_\m,K)$-module for each maximal ideal $\m$ of $R$.
\end{proposition}
\begin{proof}
By \ref{localization at a maximal of diff admissible alg}, we have $\mathcal{D}(R_{\m},K)\cong R_{\m}\otimes_R \mathcal{D}(R,K)$. 

Now
\begin{align*}
     \mathcal{D}(R_{\m},K)\otimes_{\mathcal{D}}M &= ( R_{\m}\otimes_R \mathcal{D}(R,K))\otimes_{\mathcal{D}} M \\&=  R_{\m}\otimes_R (\mathcal{D}(R,K)\otimes_{\mathcal{D}} M)\\&= R_\m\otimes_R M
\end{align*}
The structure of $\mathcal{D}(R_{\m},K)$ on $\mathcal{D}(R_{\m},K)\otimes_{\mathcal{D}}M$ gives the desired action on $R_\m\otimes_R M$. The fact that $M$ is holonomic $\D$-module if and only if $M_\m$ is holonomic as $\mathcal{D}(R_\m,K)$-module for each maximal ideal $\m$ of $R$ follows from \cite[Proposition 3.2.1.7]{Mac-14}.
\end{proof}
 We review a result due to M. Dorreh \cite[Lemma 3.2]{MD-16}.
\begin{lemma}\label{Dorreh result regarding completion}
     Let $(R,\m)$ be a regular local ring of dimension $d$ which contains a field of characteristic $0$. Assume that $P\in \Spec(R)$ such that $\height P\leq d-2$. Then $E(R/P)\otimes_R \widehat{R}$ is not holonomic $\D$-module.
\end{lemma}
As a consequence of this Lemma we prove that $E(R/P)$ is not holonomic when $R$ is a differentiably admissible $K$-algebra of dimension $d$ and $P\in \Spec(R)$ with $\height P\leq d-2$.
\begin{lemma}\label{E is not holo for diff admiss alg}
    Let $R$ be a differentiably admissible $K$-algebra of dimension $d$. If $P\in \Spec(R)$ is such that $\height P\leq d-2$, then $E_R(R/P)$ is not holonomic $\D$-module.
\end{lemma}
\begin{proof}
    Assume that $E_R(R/P)$ is a holonomic $\D(R,K)$-module for some $P\in \Spec(R)$ such that $\height P\leq d-2$.  Let $\m$ be a maximal ideal of $R$ containing $P$. Then by Proposition \ref{localization preserves holonomicity}, $E(R_\m/PR_\m)=E_R(R/P)\otimes_R R_\m$ is holonomic as $\D(R_\m,K)$-module. Again by Proposition \ref{completion preserves holonomicity}, $E(R_\m/PR_\m)\otimes_{R_\m}\widehat{R_\m}$ is holonomic as $\D(\widehat{R_\m},F)$-module where $\widehat{R_\m}=F[[x_1,\ldots,x_d]]$ by Cohen structure theorem. Since $R$ is a differentiably admissible $K$-algebra, height of each maximal ideal of $R$ is same and hence $\dim R_\m=\height \m=\dim R$.  Since $\height (PR_\m)=\height(P)\leq d-2$, we arrive at a contradiction by Lemma \ref{Dorreh result regarding completion}.
    \end{proof}
As an immediate corollary, we obtain the following result.
\begin{corollary}\label{height id bigger than equal to d-1}
     Let $R$ be a differentiably admissible $K$-algebra of dimension $d$. Let $\mathcal{T}$ be a Lyubeznik functor on $\Mod(R)$ and $\injdim_R \mathcal{T}(R)=c$. If $\mu_c(P,\mathcal{T}(R))>0$ for some prime ideal $P$ of $R$, then $\height P\geq d-1$.
\end{corollary}
\begin{proof}
Set $M=\mathcal{T}(R)$. Note that $\mu_c(P,M)=\mu_0(P,H^c_P(M))$, see \cite[Lemma 1.4, Theorem 3.4 (b)]{Lyu-93}. If possible assume that $\height P\leq d-2$. If $P\notin \Ass H^c_P(M)$, then $\mu_0(P, H^c_P(M))=0$, a contradiction. Since by Theorem \ref{finite associated prime for diff adm alg}, $\Ass_R\mathcal{T}(R)$ is finite, let us assume that $\Ass H^c_P(M)=\{P,Q_1,\ldots,Q_s\}$ where $P\subsetneq Q_i$ for all $i$. Consider the exact sequence
 $$0\rightarrow \Gamma_{Q_1\cdots Q_s}(H^c_P(M))\rightarrow H^c_P(M)\rightarrow N\rightarrow 0.$$
 Note that $N=\mathcal{F}(R)$ for some Lyubeznik functor $\mathcal{F}$ and $\Ass N=\{P\}$. Let $f\in R\setminus P$. Since $\injdim_R M=c$, we have the following exact sequence 
 $$H^c_{(P,f)}(M)\rightarrow H^c_P(M)\rightarrow H^c_P(M)_f\rightarrow H^{c+1}_{(P,f)}(M)=0.$$ This implies that the natural map $ H^c_P(M)\rightarrow H^c_P(M)_f$ is surjective and hence the natural map from $N\rightarrow N_f$ is surjective. Moreover, the fact that $\Ass N=\{P\}$ and $f\in R\setminus P$ implies the map $N\rightarrow N_f$ is injective as well. Thus, $N=N_f$ for all $f\in R\setminus P$. Therefore, $N=N_P=\mathcal{F}(R_P)=E(R_P/PR_P)^t=E(R/P)^t$ for some finite $t>0$, see \cite[Theorem 3.4 (b)]{Lyu-93}. This implies that $E(R/P)$ is holonomic as $\D$-module. This is a contradiction by Lemma \ref{E is not holo for diff admiss alg}.
\end{proof}
Having established the antecedent lemmas and propositions, we are now ready to present the proof of \hyperref[P is maximal ideal at inj dim]{Theorem A}. For the reader's convenience, we restate the theorem below.
\begin{theorem}\label{P is maixmal ideal main th}
Let $R$ be a regular affine domain over a field $K$ of characteristic $0$. Let $\mathcal{T}$ be a Lyubeznik functor on $\Mod(R)$ and $\injdim_R \mathcal{T}(R)=c$. If $\mu_c(P,\mathcal{T}(R))>0$ for some prime ideal $P$ of $R$, then $P$ is a maximal ideal of $R$.
\end{theorem}
\begin{proof}
Set $M=\mathcal{T}(R)$. Note that $\mu_c(P,M)=\mu_0(P,H^c_P(M))$, see \cite[Lemma 1.4, Theorem 3.4 (b)]{Lyu-93}. If possible let $P$ is not a maximal ideal of $R$. Therefore, $\height P=d-1$ by Corollary \ref{height id bigger than equal to d-1}. 

If $P\notin \Ass H^c_P(M)$, then $\mu_0(P, H^c_P(M))=0$, a contradiction. Therefore, let us assume that $\Ass H^c_P(M)=\{P,\m_1,\ldots,\m_s\}$ where each $\m_i$ is a maximal ideal of $R$. Consider the exact sequence
 $$0\rightarrow \Gamma_{\m_1\cdots \m_s}(H^c_P(M))\rightarrow H^c_P(M)\rightarrow N\rightarrow 0.$$
 Note that $N=\mathcal{F}(R)$ for some Lyubeznik functor $\mathcal{F}$ and $\Ass N=\{P\}$. Let $f\in R\setminus P$. Since $\injdim_R M=c$, we have the following exact sequence
 $$H^c_{(P,f)}(M)\rightarrow H^c_P(M)\rightarrow H^c_P(M)_f\rightarrow H^{c+1}_{(P,f)}(M)=0.$$ This implies that the natural map $ H^c_P(M)\rightarrow H^c_P(M)_f$ is surjective and hence the natural map from $N\rightarrow N_f$ is surjective. Since  $\Ass N=\{P\}$ and $f\in R\setminus P$, the map $N\rightarrow N_f$ is injective as well. Since this is true for all $f\in R\setminus P$, $N=N_P=\mathcal{F}(R_P)=E(R_P/PR_P)^t=E(R/P)^t$ for some finite $t>0$, see \cite[Theorem 3.4 (b)]{Lyu-93}. Thus, $E(R/P)$ is a holonomic as $\D$-module. Consider the following exact sequence in the category of $\D$-modules $$
    0\rightarrow H^{d-1}_P(R)\rightarrow E(R/P)\rightarrow \bigoplus_{P\subseteq \m}E(R/\m)\rightarrow 0.
    $$
    Since $E(R/P)$ is holonomic as $\D$-module, $\bigoplus_{P\subseteq \m}E(R/\m)$ is also holonomic as $\D$-module. Therefore, there are only finitely many maximal ideals of $R$ containing $P$. Since $R/P$ is a finitely generated $K$-algebra of dimension $1$, by  Noether normalization we have an integral extension $K[z]\hookrightarrow R/P$. Now $K[z]$ has infinitely many maximal ideals, so does $R/P$.  This implies that there are infinitely many maximal ideals of $R$ containing $P$. This leads to a contradiction.
\end{proof}
As a consequence of the theorem just proved,  we now prove \hyperref[Corollary about inj dim]{Corollary B} which shows that the equality holds of the injective dimension and the support dimension $\mathcal{T}(R)$ where $R$ is a regular affine domain.
\begin{corollary}\label{proof of the corollary}
Let $R$ be a regular affine domain over a field $K$ of characteristic $0$. Let $\mathcal{T}$ be a Lyubeznik functor on $\Mod(R)$. Then $\injdim_R\mathcal{T}(R)=\dim_R(\Supp_R\mathcal{T}(R))$.    
\end{corollary}
\begin{proof}
We use induction on $s=\dim_R(\Supp_R\mathcal{T}(R))$. Let $t=\dim R$. If $s=0$, there is nothing to prove. Assume that we have proved the result when $\dim_R(\Supp_R \mathcal{G}(R))=s-1$ for all regular affine domain over a field of characteristic $0$ and all Lyubeznik functor $\mathcal{G}$ on $\Mod(R)$. 

Let $s\geq1$. By \ref{finite associated prime for diff adm alg}, $\Ass_R\mathcal{T}(R)$ is finite. Let $P_1,\ldots,P_k$ be primes in $\Supp_R\mathcal{T}(R)$ such that $\dim(R/P_i)=s$ for all $i$. Since $R$ is a finitely generated $K$-algebra, by Noether normalization, there are algebraically independent variables $x_1\ldots,x_t$ such that $R$ is a finite $K[x_1,\ldots,x_t]$-module and a linear combination $y$ of $x_1,\ldots,x_t$ with $K[y]\cap P_i=0$ for all $i$. Then, $S^{-1}R$ where $S=K[y]\setminus \{0\}$ is a regular affine domain over $K(y)$ of dimension $t-1$. By \cite[Lemma 3.1]{Lyu-93}, there exists a Lyubeznik functor $\mathcal{F}$ on $\Mod (S^{-1}R)$ such that $\mathcal{F}(S^{-1}R)\cong S^{-1}\mathcal{T}(R)$. Also since $\dim_{S^{-1}R}(\Supp_{S^{-1}R}\left(S^{-1}\mathcal{T}(R)\right))=s-1$, by the induction hypothesis $\injdim_{S^{-1}R} \left(S^{-1}\mathcal{T}(R)\right)=s-1$. Let $Q'$ be a maximal ideal of $S^{-1}R$ such that $\mu_{s-1}(Q',S^{-1}\mathcal{T}(R))>0$. Let $Q$ be the contraction of $Q'$ in $R$. Then, $\mu_{s-1}(Q,\mathcal{T}(R))>0$ implies $\injdim_R \mathcal{T}(R)\geq s-1$. Again, since $\height Q\leq t-1$, we get $\injdim_R \mathcal{T}(R)= s$ by Theorem \ref{P is maixmal ideal main th}.
\end{proof}
\section{A lower bound on injective dimension}
In this section, we aim to prove  \hyperref[lower bound on injective dim]{Theorem C}. Before that we establish several preparatory results that are needed for our proofs. We extend some of the results of Zhang and Switala \cite{SZ-19} to the setting of differentiably admissible $K$-algebras and holonomic $\D$-modules mostly adapting their proof techniques. We begin with the following lemma.
\begin{lemma}\label{bounded inj dim}
     Let $R$ be a regular $K$-algebra where $K$ is a field of characteristic $0$. Let $\mathcal{T}$ be a Lyubeznik functor on $\Mod(R)$ and let $S$ be a multiplicative closed set of $R$. Then, $$\injdim_{S^{-1}R} S^{-1}\mathcal{T}(R)\leq \dim_{S^{-1}R} S^{-1}\mathcal{T}(R).$$
\end{lemma}
\begin{proof}
  By \cite[Lemma 3.1]{Lyu-93}, there is a Lyubeznik functor $\mathcal{F}$ on $\Mod (S^{-1}R)$ such that $\mathcal{F}(S^{-1}R)\cong S^{-1}\mathcal{T}(R)$. Note that $S^{-1}R$ is also a regular $K$-algebra. Therefore, the result follows from \cite[Theorem 3.4 (b)]{Lyu-93}. 
\end{proof}
The next lemma that we prove in the sequel is as follows.
\begin{lemma}\label{holonomic over localization}
    Let $R$ be a differentiably admissible $K$-algebra where $K$ is a field of characteristic $0$. Let $\mathcal{T}$ be a Lyubeznik functor on $\Mod(R)$. Let $S$ be a multiplicative closed set of $R$ with $\injdim_{S^{-1}R} S^{-1}\mathcal{T}(R)=t$. If $S^{-1}P\in \Spec (S^{-1}R)$ with $\mu_t(S^{-1}P,S^{-1}\mathcal{T}(R))>0$,  then $E_{S^{-1}R}(S^{-1}R/S^{-1}P)$ is holonomic as $\D(S^{-1}R,K)$-module.
    \end{lemma}
    \begin{proof}
Set $M=\mathcal{T}(R)$. Note  that $(H^t_{S^{-1}P}(S^{-1}M))_{S^{-1}P}=H^t_{PR_P}(M_P)=(H^t_P(M))_P$.  Lemma \ref{bounded inj dim} implies that  $(H^t_{S^{-1}P}(S^{-1}M))_{S^{-1}P}$ is injective as $R_P$-module. Therefore, $\mu_t(S^{-1}P,S^{-1}M)=\mu_0(S^{-1}P,H^t_{S^{-1}P}(S^{-1}M))$, see \cite[Lemma 1.4]{Lyu-93}. Since, $H^t_P(M)$ is holonomic $\D(R,K)$-module, it has finitely  many associated primes. Therefore, $H^t_{S^{-1}P}(S^{-1}M)=S^{-1}(H^t_P(M))$ has finitely many associated primes. Let $$\Ass_{S^{-1}R}H^t_{S^{-1}P}(S^{-1}M)=\{S^{-1}P,Q_1,\ldots,Q_m\}$$
Consider the exact sequence
 $$0\rightarrow \Gamma_{Q_1\cdots Q_m}(H^t_{S^{-1}P}(S^{-1}M))\rightarrow H^t_{S^{-1}P}(S^{-1}M)\rightarrow N\rightarrow 0.$$
 Note that $N=\mathcal{F}(S^{-1}R)$ for some Lyubeznik functor $\mathcal{F}$ and $\Ass N=\{S^{-1}P\}$. Let $f\in S^{-1}R\setminus S^{-1}P$. Since $\injdim S^{-1}M=t$, we have the following exact sequence 
 $$H^t_{(S^{-1}P,f)}(S^{-1}M)\rightarrow H^t_{S^{-1}P}(S^{-1}M)\rightarrow H^t_{S^{-1}P}(S^{-1}M)_f\rightarrow H^{t+1}_{(P,f)}(S^{-1}M)=0.$$ This implies that the natural map $  H^t_{S^{-1}P}(S^{-1}M)\rightarrow H^t_{S^{-1}P}(S^{-1}M)_f$ is surjective and hence the natural map from $N\rightarrow N_f$ is surjective. Moreover, the fact that $\Ass N=\{S^{-1}P\}$ and $f\in S^{-1}R\setminus S^{-1}P$ implies the map $N\rightarrow N_f$ is injective as well. Thus, $N=N_f$ for all $f\in S^{-1}R\setminus S^{-1}P$. Therefore, $N=N_{S^{-1}P}=\mathcal{F}_{R_P}(R_P)=E_{R_P}(R_P/PR_P)^t=E_{S^{-1}R}(S^{-1}R/S^{-1}P)^t$ for some finite $t>0$, see \cite[Theorem 3.4 (b)]{Lyu-93}. This implies that $E_{S^{-1}R}(S^{-1}R/S^{-1}P)$ is holonomic as $\D(S^{-1}R,K)$-module.
    \end{proof}
     The next proposition, proved in \cite[Proposition 5.4(b)]{SZ-19} for when $R$ is a power series ring over a field $K$ of characteristic $0$, also holds in the following setup, with the proof being identical in every detail.
\begin{proposition}\label{primes contained in finite no of ideals}
    Let $(R,\m)$ be a differentiably admissible $K$-algebra where $K$ is a field of characteristic $0$. Let $\mathcal{T}$ be a Lyubeznik functor on $\Mod(R)$. Let $S$ be a multiplicative closed set of $R$ with $\injdim_{S^{-1}R} S^{-1}\mathcal{T}(R)=t$. If $S^{-1}P\in \Spec (S^{-1}R)$ with $\mu_t(S^{-1}P,S^{-1}\mathcal{T}(R))>0$,  then $S^{-1}P$ is contained in finitely many distinct prime ideals of $S^{-1}R$.
 \end{proposition}
 \begin{proof}
     By Lemma \ref{holonomic over localization},  $E_{S^{-1}R}(S^{-1}R/S^{-1}P)$ is holonomic as $\D(S^{-1}R,K)$-module and hence it has finite length as $\D(S^{-1}R,K)$-module. The remainder of the proof follows verbatim from \cite[Proposition 5.4(b)]{SZ-19}, and therefore we omit the details.
 \end{proof}
 The following proposition is pivotal for establishing the main result of this section 5.
 \begin{proposition}\label{Prop about inj dim of localization}
     Let $(R,\m)$ be a differentiably admissible $K$-algebra where $K$ is a field of characteristic $0$. Let $\mathcal{T}$ be a Lyubeznik functor on $\Mod(R)$. Let $Q$ be a prime ideal of $R$ belonging to $\Supp_R\mathcal{T}(R)$, and let $f\in QR_Q$ be an element that does not belong to any minimal prime of $\mathcal{T}(R)_Q$. Then, $$\injdim_{(R_Q)_f}\mathcal({T}(R)_{Q})_f=\dim\Supp_{(R_Q)_f}\mathcal({T}(R)_{Q})_f.$$
 \end{proposition}
\begin{proof}
    Combining Lemma \ref{holonomic over localization} with Proposition \ref{primes contained in finite no of ideals}, one finds that a careful reading of the proof of \cite[Proposition 6.1]{SZ-19} shows that exactly the same argument applies to our theorem.
\end{proof}
Now we prove \hyperref[lower bound on injective dim]{Theorem C}. We restate the theorem again for the convenience of the reader.
\begin{theorem}\label{proof of bound on inj dim}
     Let $R$ be a differentiably admissible $K$-algebra where $K$ is a field of characteristic $0$. Let $\mathcal{T}$ be a Lyubeznik functor on $\Mod(R)$. Then $$\injdim_R\mathcal{T}(R)\geq \dim\Supp_R \mathcal{T}(R)-1.$$
\end{theorem}
\begin{proof}
    We may assume that $ \dim_R(\Supp_R \mathcal{T}(R))\geq 2$, since otherwise we have nothing to prove.  First, we assume that $(R,\m)$ is local. By \ref{finite associated prime for diff adm alg}, $\Ass_R\mathcal{T}(R)$ is finite. By prime avoidance, we can choose $f\in \m$ that does not belong to any minimal prime of $\mathcal{T}(R)$. Therefore, $\dim_{R_f}(\Supp_{R_f}\mathcal{T}(R)_f)=\dim\Supp_{R}\mathcal{T}(R)-1$. By Proposition \ref{Prop about inj dim of localization} (applied to $Q=\m)$, we have $\dim_{R_f}(\Supp_{R_f}\mathcal{T}(R)_f)=\injdim_{R_f}\mathcal{T}(R)_f$. Since $\injdim_R\mathcal{T}(R)\geq\injdim_{R_f}\mathcal{T}(R)_f$, the result follows.
    
    If $R$ is not local, we choose a maximal ideal $\n$ of  $R$ such that  $\dim_{R_\n}(\Supp_{R_\n}\mathcal{T}(R)_\n)=\dim_R(\Supp_{R}\mathcal{T}(R))$. Note that $R_\n$ is differentiably admissible $K$-algebra and by \cite[Lemma 3.1]{Lyu-93}, $\mathcal{T}(R)_\n=\mathcal{F}(R_\n)$ for some Lyubeznik functor $\mathcal{F}$   on $\Mod(R_\n)$. Therefore from the local case,
    $$\dim_R(\Supp_{R}\mathcal{T}(R))-1=\dim_{R_\n}(\Supp_{R_\n}\mathcal{T}(R)_\n)-1\leq \injdim_{R_\n}\mathcal{T}(R)_\n\leq \injdim_{R}\mathcal{T}(R).$$
\end{proof}

\section{Finiteness of associated primes}
In this section, we prove \hyperref[finiteness of associated primes]{Theorem D}. First, we recall a theorem \cite[Theorem 23.2]{Mat-86} from Matsumura's book. It should be noted that there is a typographical error in the book regarding this theorem. The correct statement is  as follows.
\begin{theorem}\label{Matsumura theorem for associated primes}
    If $\phi: A\rightarrow B$ is a ring homomorphism and $B$ is flat over $A$, then for any $A$-module $M$,
    \begin{enumerate}[\rm (1)]
    \item $\Ass_A\left(B/PB\right)={Q}$ if $Q\in \Spec A$.
        \item $Ass_B\left(M\otimes_AB\right)=\bigcup_{P\in \Ass_AM} \Ass_B\left(B/PB\right).$
    \end{enumerate}
    
\end{theorem}
Now we prove \hyperref[finiteness of associated primes]{Theorem D} which we restate again.
\begin{theorem}\label{proof of finite ass prime for diff ad alg}
    Let $R$ be a differentiably admissible $K$-algebra and $S$ be the polynomial ring $R[x_1,\ldots,x_m]$. If $I$ is an ideal of $S$, then the set $\Ass_S H^i_I(S)$ is finite.
\end{theorem}
\begin{proof}
    We divide the proof in two cases.\medskip
    
    \noindent
\textbf{Case 1:} $I$ is a homogeneous ideal.\medskip

Let $I$ be a homogeneous ideal in $S=R[x_1,\ldots,x_m]$. Consider the flat extension $S\rightarrow A=R[[x_1,\ldots,x_m]]$. Since $R$ is a differentiably admissible algebra, $R[[x_1,\ldots,x_m]]$ is also a differentiably admissible algebra. Therefore, by Theorem \ref{finite associated prime for diff adm alg}, $\Ass_AH^i_{IA}(A)$ is a finite set. Note that $H^i_I(S)\otimes_S A= H^i_{IA}(A)$.
By Theorem \ref{Matsumura theorem for associated primes},
\begin{align}\label{equation 1}
    \Ass_A(H^i_I(S)\otimes_SA)=\bigcup_{P\in \Ass_S H^i_I(S)}\Ass_A\left(A/PA\right)
\end{align}
We note that if $P\in \Ass_S H^i_I(S)$, then $\Ass_A\left(A/PA\right)\neq \phi$. This is because if $P\in \Ass_S H^i_I(S)$, then $P$ is a graded prime ideal of $S$. Therefore, $P\subseteq(\m,(x_1,\ldots,x_m))$ for some $\m$ maximal ideal of $S$. Let us denote the ideal $(\m,(x_1,\ldots,x_m))$ by $\n$.  Hence, there is a surjection $S/P\rightarrow S/\n$, i.e., we have an exact sequence $S/P\rightarrow S/\n\rightarrow 0$. Tensoring this sequence with $A$ gives us the exact sequence $ A/PA\rightarrow S/\n\otimes_S A\rightarrow 0$. Since $S/\n\otimes_S A \neq 0$, we get $A/PA\neq 0$ and hence $\Ass_A\left(A/PA\right)\neq \phi$. We also note that if $Q\in \Ass_A\left(A/PA\right)$, then $Q\cap S=P$. Therefore, if $P_1, P_2\in  \Ass_S H^i_I(S) $ such that $P_1\neq P_2$, then $\Ass_A\left(A/P_1A\right)\cap \Ass_A\left(A/P_2A\right)=\phi$. Therefore, from equation \ref{equation 1} we conclude that the set $\Ass_S H^i_I(S)$ is finite since $\Ass_A(H^i_I(S)\otimes_SA)$ is finite.\medskip 

 \noindent
\textbf{Case 2:} $I$ is not a homogeneous ideal.\medskip

Let $T=S[z]$ where $S=R[x_1,\ldots,x_m]$ with $\deg(z)=\deg(x_i)=1$ for all $i$.  Therefore, $S=T/(z-1)$. Let $\pi:T\rightarrow S$ be the natural homomorphism. We also have the localization map from $T\rightarrow T_z$. Then $\pi$ factors in a natural way through a homomorphism $\phi: T_z\rightarrow S$. Therefore, we have the following commutative diagram
\begin{equation*}
    \begin{tikzcd}
        T \arrow{d}{} \arrow{r}{\pi} & S \\
        T_z \arrow[dashed]{ru}{\phi}
    \end{tikzcd}
\end{equation*}
By \cite[Proposition 1.5.18]{Bruns}, $T_z\cong \left(T_z\right)_0[z,z^{-1}]$ and $ \left(T_z\right)_0\cong S$. Let $J$ be the homogenization of the ideal $I$ i.e., $J=IT_z\cap T$. Note that $IT_z=JT_z$. By case 1, $\Ass_TH^i_J(T)$ is a finite set and hence $\Ass_{T_z}H^i_{JT_z}(T_z)$ is finite set. Now
\begin{align*}
    \Ass_{T_z}H^i_{JT_z}(T_z)&=\Ass_{T_z}H^i_{IT_z}(T_z)\\&=\Ass_{S[z,z^{-1}]}H^i_{JS[z,z^{-1}]}(S[z,z^{-1}])\\&=\Ass_{T_z}(H^i_I(S)\otimes_ST_z).
\end{align*}
Note that $T_z=S[z,z^{-1}]$ is a faithfully flat extension of $S$ and if $P$ is a prime ideal of $S$, then $Tz/PT_z=(S/P)[z,z^{-1}]$ is non zero. Therefore, a similar argument as in Case 1 yields that $\Ass_S H^i_I(S)$ is finite.
\end{proof}

\section*{Acknowledgement}
	I would like to express my sincere gratitude to Prof. Tony J. Puthenpurakal, my PhD supervisor, for his constant guidance, encouragement, and for suggesting this problem. I am also thankful for his careful reading of the manuscript and valuable feedback. I thank the Government of India for support through the Prime Minister's Research Fellowship (PMRF ID: 1303161).

\providecommand{\bysame}{\leavevmode\hbox to3em{\hrulefill}\thinspace}
\providecommand{\MR}{\relax\ifhmode\unskip\space\fi MR }
\providecommand{\MRhref}[2]{
  \href{http://www.ams.org/mathscinet-getitem?mr=#1}{#2}
}

\end{document}